\documentclass[12pt]{amsart}
\usepackage[all]{xy}
\usepackage{amssymb}
\usepackage{amsthm}
\usepackage{amsmath}
\usepackage{amscd,enumitem}
\usepackage{verbatim}
\usepackage{eurosym}
\usepackage{graphicx}
\usepackage{dsfont}
\usepackage{stmaryrd}
\usepackage{color}
\usepackage{float}
\usepackage{longtable}
\usepackage{dcolumn}
\usepackage[mathscr]{eucal}
\usepackage[all]{xy}
\usepackage{hyperref}
\usepackage[title]{appendix}
\usepackage[usenames,dvipsnames]{xcolor}
\usepackage{bbm}
\usepackage[textheight=8.85in, textwidth=6.8in]{geometry}
\usepackage{multirow}
\usepackage{caption}

\theoremstyle{plain}
\newtheorem{theorem}{Theorem}[section]
\newtheorem{corollary}[theorem]{Corollary}
\newtheorem{lemma}[theorem]{Lemma}
\newtheorem{proposition}[theorem]{Proposition}

\theoremstyle{definition}
\newtheorem{definition}[theorem]{Definition}

\theoremstyle{remark}

\newcommand{\F}{\mathbb{F}}

\newcommand{\Cl}{\mathrm{Cl}}
\newcommand{\Leg}{\mathrm{Leg}}

\catcode`,\active

\catcode`\,12

\numberwithin{equation}{section}

\begin{document}
\title[Distribution of rational points of an algebraic surface over finite fields]
{Distribution of rational points of an algebraic surface over finite fields}

\author{Sudhir Pujahari}
\address{School of Mathematical Sciences, National Institute of Science Education and Research, Bhubaneswar, An OCC of Homi Bhabha National Institute,  P. O. Jatni,  Khurda 752050, Odisha, India.}
\email{spujahari@niser.ac.in}
\author{Neelam Saikia}
\address{Department of Mathematics, Indian Institute of Technology Bhubaneswar, Odisha, India.}
 \email{neelamsaikia@iitbbs.ac.in}


\begin{abstract}  
The number of points on a certain one parameter family of algebraic surface over a finite field $\F_p$ can be expressed as $p^2+A_p(\lambda),$ where $A_p(\lambda)$ is a character sum and $\lambda$ is an element of the finite field $\F_p.$ In this paper, we study the distribution of the term $A_p(\lambda)$ as the surface varies over a large family of algebraic surfaces of fixed genus and growing $p.$ The power moments of $A_p$'s are weighted sums of Catalan numbers. As a consequence of these results, we obtain limiting distributions of certain families of hypergeometric functions over large finite fields. 
\end{abstract}

\maketitle

\section{Statements of main theorems}
\noindent The main goal of is this paper is to study certain statistical questions arising from arithmetic geometry. 
 Let $p$ be an odd prime and $\F_p$ be a finite field with $p$ elements. 
For $\lambda\in\F_p\setminus\{0\},$ define a surface $X_{\lambda}$ as given by
$$
X_{\lambda}:\ \ z^2=xy(1+x+y)\left(xy+(1-\lambda)/\lambda^2\right).
$$ 

\noindent Then, the number of $\F_p$-rational points of the surface $X_{\lambda}$ can be written as
$$
|X_\lambda(\F_p)|=p^2+\sum\limits_{x,y\in\F_p}\phi\left(xy(1+x+y)\left(xy+(1-\lambda)/\lambda^2\right)\right),
$$ 

\noindent where $\phi$ is the quadratic character of $\F_p$ (with the convention that $\phi(0)=0.$) In this paper, we study the distribution of the number of $\F_p$-rational points of the surface $X_{\lambda}$ when the surface is drawn at random from the family of surfaces $X_\lambda$ as $\lambda$ varies over $\F_p\setminus\{0\}$ for large $p.$
Now, for $\lambda\in\F_p\setminus\{0\},$ let $A_p(\lambda)$ be defined by
$$
A_p(\lambda):=p^2-|X_\lambda(\F_p)|=-\sum\limits_{x,y\in\F_p}\phi\left(xy(1+x+y)\left(xy+(1-\lambda)/\lambda^2\right)\right).
$$
\noindent Then, the problem of studying statistical distribution of the number of rational points of $X_{\lambda}$ is equivalent to examine the distribution of the character sum $A_p(\lambda)$ as $\lambda$ varies over $\F_p$ in the limit as $p$ grows. Investigation of such questions for curves and surfaces in various settings are central problems in arithmetic geometry. For instance, a classical theorem of Birch \cite{Birch} established the semicircular distribution of Frobenius traces elliptic curves over finite fields in vertical setting. In a difficult setting (called as horizontal setting, probably the most difficult setting), Clozel, Harris, Shepherd-Barron and Taylor \cite{Taylor} have considered a fixed elliptic curve under some conditions and settled semicircular distribution for Frobenius traces of the fixed elliptic curve taken over all primes $p$. This was a conjecture of Sato and Tate formulated independently around 1960's. Along similar directions for abelian varieties analogous results have been discovered by Fit\'e, Kedlaya, and Sutherland (for more details see \cite{Sutherland}). In our purpose, we formulate the problem in Birch's setting. Recently, in the similar formulation, Ono, Saad and the second author \cite{KHN} investigated the limiting distributions of traces of Frobenius of Legendre's families of elliptic curves and a certain family of $K3$ surfaces over large finite fields. They recovered semicircular distribution for the Legendre's family of elliptic curves, whereas for the $K3$ surfaces their established distribution is named as Batman distribution or that of traces of the real orthogonal matrices $O_3.$ In view of these consequences, it is natural to seek further examples of surfaces over finite fields and investigate similar questions. In this paper, our goal is to consider the surface $X_\lambda$ over finite fields and conclude a new distribution for the character sum $A_p(\lambda)$ that is associated to the number of points of $X_\lambda(\F_p).$  In the following theorem, we first compute the moments of $A_p(\lambda).$ 

\begin{theorem}\label{Frobenius-Moments}
Let $p>3$ be a prime and $m$ be a fixed positive integer. Then as $p\rightarrow\infty$ we have 

\begin{align}
\sum\limits_{\lambda\in\F_p\setminus\{0\}}A_p(\lambda)^{m}=\sum\limits_{i=0}^{m}(-1)^{i}{m\choose i}\frac{(2i)!p^{m+1}}{(i)!(i+1)!}+o_m(p^{m+1}).\notag
\end{align}
\end{theorem}

\noindent The moments that appear in Theorem \ref{Frobenius-Moments} can be described as weighted sums of Catalan numbers. The even moments of $A_p(\lambda)$ are matching with the even moments of traces of the real orthogonal group $O_3$ and that of the $K3$ surfaces (for more details, see \cite{O3} and \cite{KHN}). If we consider the normalized values $\frac{A_p(\lambda)}{p}\in[-3,3]$ as random variables on $\F_p,$ then using the moments obtained in Theorem \ref{Frobenius-Moments}, we examine the behaviour of $A_p(\lambda)$ as $p\rightarrow\infty$ in the next corollary.

\begin{corollary}\label{Frobenius-Distribution}
Let $p>3$ be a prime and $-3\leq a<b\leq 3$ be two real numbers. Then as $p\rightarrow\infty$ we have 

$$
\lim_{p\rightarrow \infty}\frac{|\lambda\in\F_p: \frac{A_p(\lambda)}{p}\in[a,b]|}{p}=\frac{1}{2\pi}\int\limits_{a}^b f(t)dt,
$$
where 
$$
f(t)=\begin{cases}
\sqrt{\frac{3-t}{1+t}}, & \ \text{if}\ -1<t<3;\\
0, & \text{otherwise.}
\end{cases}
$$
\end{corollary}

\noindent Let $\widehat{\mathbb{F}_p^\times}$ be the group of all multiplicative characters of $\mathbb{F}_p^{\times}$. Let $\overline{\chi}$ denote the inverse of a multiplicative character $\chi$ of $\F_p$ (with the convention that $\chi(0)=0).$ 
For multiplicative characters $\chi$ and $\psi$ of $\mathbb{F}_p,$ recall that the Jacobi sum $J(\chi,\psi)$ is defined by
\begin{align}
J(\chi,\psi):=\sum_{y\in\mathbb{F}_p}\chi(y)\psi(1-y).\notag
\end{align}
Moreover, the normalized Jacobi sum also known as binomial coefficient is given by
$$
{\chi\choose \psi}:=\frac{\psi(-1)}{p}J(\chi,\overline{\psi}).
$$
Using these binomials, Greene \cite{greene} introduced a class of functions over finite fields known as Gaussian hypergeometric function. Along with number theoretic importance, these functions can be naturally described as a finite field analogue of classical hypergeometric functions (for example, see \cite{greene, evans2, EG}) etc. Recently, Ono, Saad and the second author \cite{KHN} initiated a study of limiting distributions of certain families of Gaussian hypergeometric functions over finite fields and concluded semicircular distributions and that of traces of the real orthogonal group $O_3.$ Along similar lines, it is natural to consider further hypergeometric functions in different settings and propose similar questions about their distributions. In our purpose, we focus on hypergeometric functions in the $p$-adic setting introduced by McCarthy \cite{mccarthy-pacific}. Recently, in \cite{SN}, the authors evidenced semicircular distributions for certain families of $p$-adic hypergeometric functions over large finite fields. In this paper, we further develop distributional results for two more families of $p$-adic hypergeometric functions over large finite fields in the setting of Birch. These two families of  $p$-adic hypergeometric functions appear naturally in the point counting formulas of the surface $X_{\lambda}$ and as results of Theorem \ref{Frobenius-Moments} and Corollary \ref{Frobenius-Distribution} we derive distributional results for these two families of functions. To give the precise statements of the results we now fix some notation.

Let $\mathbb{Z}_p$ denote the ring of $p$-adic integers and $\mathbb{Q}_p$ denote the field of $p$-adic numbers.
Let $\Gamma_p(\cdot)$ denote Morita's $p$-adic gamma function and $\omega$ denote the Teichm\"{u}ller character of $\mathbb{F}_p$ with $\omega(a)\equiv a\pmod{p}.$ For $x\in\mathbb{Q},$ $\lfloor x\rfloor$ denotes the greatest integer less than or equal to $x$ and $\langle x\rangle$ denotes the fractional part of $x$, satisfying $0\leq\langle x\rangle<1$. Using these notation, $p$-adic hypergeometric function over finite field is given as follows:

\begin{definition}\cite[Definition 5.1]{mccarthy-pacific} \label{defin1}
Let $p$ be an odd prime and $t \in \mathbb{F}_p$.
For a positive integer $n$ and $1\leq k\leq n$, let $a_k$, $b_k$ $\in \mathbb{Q}\cap \mathbb{Z}_p$.
Then 
\begin{align}
&{_n\mathbb{G}_n}\left[\begin{array}{cccc}
             a_1, & a_2, & \ldots, & a_n \\
             b_1, & b_2, & \ldots, & b_n
           \end{array}\mid t
 \right]_p:=\frac{-1}{p-1}\sum_{a=0}^{p-2}(-1)^{an}~~\overline{\omega}^a(t)\notag\\
&\times \prod\limits_{k=1}^n(-p)^{-\lfloor \langle a_k \rangle-\frac{a}{p-1} \rfloor -\lfloor\langle -b_k \rangle +\frac{a}{p-1}\rfloor}
 \frac{\Gamma_p(\langle a_k-\frac{a}{p-1}\rangle)}{\Gamma_p(\langle a_k \rangle)}
 \frac{\Gamma_p(\langle -b_k+\frac{a}{p-1} \rangle)}{\Gamma_p(\langle -b_k \rangle)}.\notag
\end{align}
\end{definition}
\noindent 

\noindent In our context, the natural choices are those $p$-adic hypergeometric functions or that of character twists whose special values are related to the number of $\F_p$-points of the surface $X_\lambda.$ 
Before providing quantitative statements of our results, we further fix some notation for brevity. For $\lambda\neq1,$ let 
$$_3{G}_3(\lambda)_p:=\psi_3\left(\frac{\lambda}{(1-\lambda)^2}\right)\cdot {_3\mathbb{G}_3}\left[\begin{matrix}
\frac{1}{3}, & \frac{1}{3}, & \frac{1}{3}\vspace{1mm}\\
\frac{1}{12}, & \frac{7}{12}, & \frac{5}{6}
\end{matrix}\mid \frac{-4\lambda}{(1-\lambda)^2}\right]_p,$$  
where  $\psi_3=\omega^{\frac{p-1}{3}}$ is a multiplicative character of orders $3.$ Then as a consequence of Theorem \ref{Frobenius-Moments}, we derive conditional moments for the hypergeometric function ${_3{G}_3}(\lambda)_p$ in the next theorem.

\begin{theorem}\label{3G3-Moments}
Let $p>3$ be a prime such that $p\equiv1\pmod{3}$ and $m$ be a fixed positive integer. Then as $p\rightarrow\infty$ we have 

\begin{align}
\sum\limits_{\lambda\in\F_p\setminus\{1\}} {_3G_3}(\lambda)_p^m=\delta(m)\cdot\sum\limits_{i=0}^{m}(-1)^{i}{m\choose i}\frac{(2i)!p^{m}}{(i)!(i+1)!}+o_m(p^{m}),\notag
\end{align}
where $\delta(m)=\begin{cases}
1, & \ \text{if\ m\ is\ even;}\\
-1, & \ \text{if\ m\ is\ odd.}
\end{cases}$
\end{theorem}
In analogy with Corollary \ref{Frobenius-Distribution}, we model the values $ {_3G_3}(\lambda)_p\in[-3,3]$ as random variables on $\F_p$ and counter the limiting distribution of the normalized  values $ {_3G_3}(\lambda)_p$ in the limit of $p.$ 

\begin{corollary}\label{3G3-Distribution}
Let $p>3$ be a prime such that $p\equiv1\pmod{3}$ and $-3\leq a<b\leq 3$ be two real numbers. Then as $p\rightarrow\infty$ we have 

$$
\lim_{p\rightarrow \infty}\frac{|\lambda\in\F_p: {_3G_3}(\lambda)_p\in[a,b]|}{p}=\frac{1}{2\pi}\int\limits_{a}^b f(t)dt,
$$
where 
$$
f(t)=\begin{cases}
\sqrt{\frac{3+t}{1-t}}, & \ \text{if}\ -3<t<1;\\
0, & \text{otherwise.}
\end{cases}
$$
\end{corollary}

\noindent Moreover, for $\lambda\in\F_p\setminus\{1\},$ let 
$${_9G_9}(\lambda)_p:={_9\mathbb{G}_9}\left[\begin{matrix}
\frac{1}{3}, & \frac{1}{3}, & \frac{1}{3}, & \frac{2}{3}, & \frac{2}{3}, & \frac{1}{3}, &0, &,0, & 0\vspace{1mm}\\
\frac{1}{12}, & \frac{5}{12}, & \frac{7}{12}, & \frac{11}{12}, & \frac{1}{6}, &  \frac{5}{6}, & \frac{1}{4}, & \frac{3}{4}, & \frac{1}{2}
\end{matrix}\mid \frac{-4^3\lambda^3}{(1-\lambda)^6}\right]_p.$$

Then the following theorem gives conditional moments for this function as a consequence of Theorem \ref{Frobenius-Moments}.

\begin{theorem}\label{9G9-Moments}
Let $p>3$ be a prime such that $p\equiv2\pmod{3}$ and $m$ be a fixed positive integer. Then as $p\rightarrow\infty$ we have 

\begin{align}
\sum\limits_{\lambda\in\F_p\setminus\{1\}} {_9G_9}(\lambda)_p^m=\delta(m)\sum\limits_{i=0}^{m}(-1)^{i}{m\choose i}\frac{(2i)!p^{m+1}}{(i)!(i+1)!}+o_m(p^{m+1});
\end{align}
where $\delta(m)=\begin{cases}
1, & \ \text{if\ m\ is\ even;}\\
-1, & \ \text{if\ m\ is\ odd.}
\end{cases}$
\end{theorem}
As a counterpart of Corollary \ref{3G3-Distribution}, we obtain the limiting distribution of the normalized values $\frac{{_9G_9}(\lambda)_p}{p}\in[-3,3]$ as $p\rightarrow\infty$ in the following result.

\begin{corollary}\label{9G9-Distribution}
Let $p>3$ be a prime such that $p\equiv2\pmod{3}$ and $-3\leq a<b\leq 3$ be two real numbers. Then as $p\rightarrow\infty$ we have 

$$
\lim_{p\rightarrow \infty}\frac{|\lambda\in\F_p: \frac{{_9G_9}(\lambda)_p}{p}\in[a,b]|}{p}=\frac{1}{2\pi}\int\limits_{a}^b f(t)dt,
$$
where 
$$
f(t)=\begin{cases}
\sqrt{\frac{3+t}{1-t}}, & \ \text{if}\ -3<t<1;\\
0, & \text{otherwise.}
\end{cases}
$$
\end{corollary}

The rest of the paper is organized as follows. In Section 2, we discuss arithmetic of Legendre and Clausen elliptic curves. In Section 3, we recall important results of Gauss and Jacobi sums, certain product formulas of $p$-adic gamma functions and Gross-Koblitz formula. In addition, we prove three propositions that can be used to establish relations between the point counting formulas of the surface $X_{\lambda}$ and the hypergeometric functions ${_3G_3}(\lambda)_p$ and 
${_9G_9}(\lambda)_p.$ In Section 4, we provide the proofs of the main theorems.

\section{Arithmetic of elliptic curves}
\noindent We begin this section by recalling Legendre and Clausen elliptic curves.
For $\lambda\neq0,1,$ let $$E^{\Leg}_{\lambda}:~ y^2=x(x-1)(x-\lambda)$$ be a Legendre elliptic curve and for
$\lambda\neq0,-1,$ let $$E^{\Cl}_{\lambda}:~ y^2=(x-1)(x^2+\lambda)$$ be a Clausen elliptic curve over $\F_p,$ where $p>3$ be a prime. Moreover, let $$a^{\Leg}_p(\lambda):=p+1-|E_{\lambda}^{\Leg}(\F_p)|\ \text{and}\ a^{\Cl}_p(\lambda):=p+1-|E_{\lambda}^{\Cl}(\F_p)|$$ be their respective traces of Frobenius. By the classical Hasse bound for the traces of Frobenius of elliptic curves we may write

\begin{align}\label{Hasse-bound}
|a^{\Leg}_p(\lambda)|\leq 2\sqrt{p}\ \text{and}\ |a^{\Cl}_p(\lambda)|\leq 2\sqrt{p}.
\end{align}

The traces of Frobenius of Legendre and Clausen elliptic curves correspond to special values of hypergeometric functions over finite fields. To give the precise statement we recall the following two Greene's hypergeometric functions over finite fields.
For $\lambda\in\F_p,$ let
$$
{_2F_1}(\lambda)_p:=\sum\limits_{\chi\in\widehat{\F_p^{\times}}}{\phi\chi\choose\chi}{\phi\chi\choose\chi}\chi(\lambda)\ \text{and}\
{_3F_2}(\lambda)_p:=\sum\limits_{\chi\in\widehat{\F_p^{\times}}}{\phi\chi\choose\chi}{\phi\chi\choose\chi} {\phi\chi\choose\chi}\chi(\lambda).
$$

Then by the work of Ono \cite{ono}, these two functions appear in the expressions of traces of Frobenius of the Legendre and Clausen elliptic curves as described below.
\begin{theorem}\label{relation-1}\cite[Theorem 1]{ono}
For $p\geq5$ and $\lambda\in\F_p\setminus\{0,1\}$ we have 
$$
{_2F_1}(\lambda)_p=\frac{-\phi(-1)}{p}a^{\Leg}_p(\lambda).
$$
\end{theorem}

\begin{theorem}\label{relation-2}\cite[Theorem 5]{ono}
For $p\geq5$ and $\lambda\in\F_p\setminus\{0,-1\}$ we have 
$$
p+p^2\phi(1+\lambda)\cdot {_3F_2}\left(\frac{\lambda}{1+\lambda}\right)_p=a^{\Cl}_p(\lambda)^2.
$$
\end{theorem}

If we use \eqref{Hasse-bound} in Theorem \ref{relation-2} then we have 

\begin{align}\label{3F2bound}
|p\cdot{_3F_2}(\lambda)_p|\leq3.
\end{align}
Furthermore, we recall the next two results that provides asymptotic moment formulas of the hypergeometric functions ${_2F_1}(\lambda)_p$ and ${_3F_2}(\lambda)_p$ in the limit of $p.$

\begin{theorem}\label{2F1-moments}\cite[Theorem 1.1]{KHN}
If $m$ is a fixed positive integer, then as $p\rightarrow\infty$ we have
$$
p^{m/2-1}\sum\limits_{\lambda\in\F_p}{_2F_1}(\lambda)_p^m=\begin{cases}
o_m(1), \ \text{if \ m \ is \ odd}\\
\frac{(2n)!}{n!(n+1)!}+o_m(1),  \  \text{if}\ m=2n\ \text{is \ even}.
\end{cases}
$$
\end{theorem}

\begin{theorem}\label{3F2-moments}\cite[Theorem 1.3]{KHN}
If $m$ is a fixed positive integer, then as $p\rightarrow\infty$ we have
$$
p^{m-1}\sum\limits_{\lambda\in\F_p}{_3F_2}(\lambda)_p^m=\begin{cases}
o_m(1), \ \text{if \ m \ is \ odd}\\
\sum\limits_{i=0}^{m}(-1)^i{m\choose i}\frac{(2i)!}{i!(i+1)!}+o_m(1),  \  \text{if \ m \ is \ even.}
\end{cases}
$$
\end{theorem}

To this end, we discuss about the fact that the elliptic curve $E^{\Cl}_{-\lambda^2}$ is a quadratic twist of the elliptic curve  $E^{\Leg}_{\alpha},$ where $\alpha=\frac{2\lambda}{\lambda-1}.$
\begin{lemma}\label{isomorphism-1}
For $\lambda\in\F_p\backslash\{0,\pm1\},$ the elliptic curves $E^{\Cl}_{-\lambda^2}: \  y^2=(x-1)(x^2-\lambda^2)$ and $E^{\Leg}_{\alpha}:\  y^2=x(x-1)(x-\alpha),$ where $\alpha=\frac{2\lambda}{\lambda-1}$ are isomorphic over $\F_p$ or $\F_{p^2}.$
\end{lemma}
\begin{proof}
If we consider the elliptic curve $E^{\Cl}_{-\lambda^2}: \ y^2=(x-1)(x-\lambda^2),$ then by taking the transformations $x\mapsto(1-\lambda)x+\lambda$ and $y\mapsto(1-\lambda)^{3/2}y$ we obtain that $E^{\Cl}_{-\lambda^2}$ is either isomorphic to or a quadratic twist of the elliptic curve
$E^{\Leg}_{\alpha}.$
\end{proof}
\subsection{Moments of Frobenius of traces of Clausen elliptic curves}
\begin{lemma}\label{clausen-moment-lemma}
For a positive integer $j,$ as $p\rightarrow\infty$ we have 
$$
\sum\limits_{\lambda\in\F_p\setminus\{0,-1\}}(a_p^{\Cl}(\lambda))^{2j}=\frac{(2j)!p^{j+1}}{j!(j+1)!}+o_j(p^{j+1}).
$$
\end{lemma}

\begin{proof}
The result can be derived by using Proposition 3.2 of \cite{KHN}, Lemma 5.4 of \cite{KHN} and Lemma 5.6 of \cite{KHN}.
\end{proof}

\begin{proposition}\label{moment-proposition} 
For a positive integer $j,$ as $p\rightarrow\infty$ we have 
$$
\sum\limits_{\substack{\lambda\in\F_p\setminus\{0\}\\  \lambda^2\neq-1}}(a_p^{\Cl}(\lambda^2))^{2j}=\frac{(2j)!p^{j+1}}{j!(j+1)!}+o_j(p^{j+1}).
$$

\end{proposition}

\begin{proof}
We prove the result by considering the following cases: First suppose that $p\equiv1\pmod{4}.$ Now, taking the transformation $\lambda^2\rightarrow-\lambda^2,$
we may write

\begin{align}\label{eqn-9}
\sum\limits_{\substack{\lambda\in\F_p\setminus\{0\}\\  \lambda^2\neq-1}}a_p^{\Cl}(\lambda^2)^{2j}=\sum\limits_{\lambda\in\F_p\setminus\{0,\pm1\}}a_p^{\Cl}(-\lambda^2)^{2j}.
\end{align}
%%%%%%%%%%%%%%%%%%%%
%%%%%%%%%%%%%%%%%%%% 

\noindent Using Lemma \ref{isomorphism-1} for $\lambda\neq0,\pm1,$ we have that 
$$|a_p^{\Cl}(-\lambda^2)|=|a_p^{\Leg}(\beta)|,\  \text{where} \ \beta=\frac{2\lambda}{\lambda-1}.$$
By making use of this relation in \eqref{eqn-9} we write

\begin{align}\label{eqn-10}
\sum\limits_{\substack{\lambda\in\F_p\setminus\{0\}\\  \lambda^2\neq-1}}(a_p^{\Cl}(\lambda^2))^{2j}=\sum\limits_{\lambda\in\F_p\setminus\{0,\pm1\}}\left(a_p^{\Leg}\left(\frac{2\lambda}{\lambda-1}\right)\right)^{2j}.
\end{align}
By taking the transformation $\lambda\mapsto\frac{\lambda}{\lambda-2}$ on the right side of \eqref{eqn-10} and using Theorem \ref{2F1-moments}, Theorem \ref{relation-1} and the Hasse bound \eqref{Hasse-bound} we conclude the result. 

Now, suppose that $p\equiv3\pmod{4}.$ Then separating squares and non-squares, we have
$$
\sum\limits_{\lambda\in\F_p\setminus\{0,-1\}}(a_p^{\Cl}(\lambda))^{2j}=\frac{1}{2}\sum\limits_{\lambda\in\F_p\setminus\{0\}}(a_p^{\Cl}(\lambda^2))^{2j}+\frac{1}{2}\sum\limits_{\lambda\in\F_p\setminus\{0,\pm1\}}(a_p^{\Cl}(-\lambda^2))^{2j}.
$$
This gives
\begin{align}\label{eqn-11}
\sum\limits_{\lambda\in\F_p\setminus\{0\}}(a_p^{\Cl}(\lambda^2))^{2j}=2\sum\limits_{\lambda\in\F_p\setminus\{0,-1\}}(a_p^{\Cl}(\lambda))^{2j}-\sum\limits_{\lambda\in\F_p\setminus\{0,\pm1\}}(a_p^{\Cl}(-\lambda^2))^{2j}.
\end{align}
Again using Lemma \ref{isomorphism-1} we have that 
$$
\sum\limits_{\lambda\in\F_p\setminus\{0,\pm1\}}(a_p^{\Cl}(-\lambda^2))^{2j}=\sum\limits_{\lambda\in\F_p\setminus\{0,\pm1\}}\left(a_p^{\Leg}\left(\frac{2\lambda}{\lambda-1}\right)\right)^{2j}.
$$
Hence, applying similar arguments as before in the case $p\equiv1\pmod{4},$ we have 

\begin{align}\label{eqn-12}
\sum\limits_{\lambda\in\F_p\setminus\{0,\pm1\}}(a_p^{\Cl}(-\lambda^2))^{2j}=\frac{(2j)!p^{j+1}}{j!(j+1)!}+o_j(p^{j+1})
\end{align}
as $p\rightarrow\infty.$ Finally, applying Lemma \ref{clausen-moment-lemma} on the first summation of the right side of \eqref{eqn-11} and using \eqref{eqn-12} on the second summation of the right side of \eqref{eqn-11} we conclude the result.

\end{proof}

\section{Character sums and $p$-adic gamma functions}
We start this section by discussing certain facts about multiplicative characters. For instance, the following result provides orthogonality relation of multiplicative characters.
\begin{lemma}\cite[Chapter 8]{ireland}
Let $p$ be an odd prime. Then we have
\begin{align}\label{orthogonal-1}
\sum_{\chi\in\widehat{\mathbb{F}_p^\times}}\chi(x)=\left\{
   \begin{array}{ll}
    p-1 , & \hbox{if $x=1$;} \\
  0, & \hbox{if $x\neq1$.}
   \end{array}
 \right.
\end{align}
\end{lemma}
\noindent Let $\zeta_p$ denote a fixed primitive $p$-th root of unity. Then, for a multiplicative character
 $\chi$ of $\mathbb{F}_p,$ the Gauss sum is defined by
\begin{align}
g(\chi):=\sum\limits_{x\in \mathbb{F}_p}\chi(x)~\zeta_p^x.\notag
\end{align}
It is easy to verify that $g(\varepsilon)=-1$. For $p$ odd prime, multiplicative character $\chi\in\widehat{\F_p^{\times}}$ and $x\in\F_p,$ we define the following functions:

\begin{align}
\Delta(p):&=\begin{cases}
1,\ \text{if}\ p\equiv1\pmod{4},\\
0,\ \text{if}\ p\equiv3\pmod{4}.
\end{cases}, ~
\delta(\chi):&=\begin{cases}
1,\ \text{if}\ \chi=\varepsilon,\\
0,\ \text{if}\ \chi\neq\varepsilon.
\end{cases} \ \text{and}\
\delta(x):&=\begin{cases}
1,\ \text{if}\ x=0,\\
0,\ \text{if}\ x\neq0.
\end{cases}\notag
\end{align}

\noindent The following product of Gauss sums is going to be used in most of our calculations.
\begin{lemma}\cite[eq. 1.12]{greene} For $\chi\in \widehat{\mathbb{F}_p^\times},$ we have
\begin{align}\label{inverse}
g(\chi)g(\overline{\chi})=p\chi(-1)-(p-1)\delta(\chi).
\end{align}
\end{lemma}

\noindent Furthermore, we recall an identity relating Gauss and Jacobi sums.
\begin{lemma}\cite[eq. 1.14]{greene}
For $\chi_1,\chi_2\in \widehat{\mathbb{F}_p^\times},$ Then
\begin{align}\label{Gauss-Jacobi-relation}
J(\chi_1,\chi_2)=\frac{g(\chi_1)g(\chi_2)}{g(\chi_1\chi_2)}+(p-1)\chi_2(-1)\delta(\chi_1\chi_2).
\end{align}
\end{lemma}

\subsection{$p$-adic gamma functions}
Let $\overline{\mathbb{Q}_p}$ denote the algebraic closure of $\mathbb{Q}_p$ and $\mathbb{C}_p$ denote the completion of $\overline{\mathbb{Q}_p}$. 
We now recall $p$-adic gamma function. For a positive integer $n,$
the $p$-adic gamma function $\Gamma_p(n)$ is defined as
\begin{align}
\Gamma_p(n):=(-1)^n\prod\limits_{0<j<n,p\nmid j}j.\notag
\end{align}
 It can be extended to all $x\in\mathbb{Z}_p$ by setting $\Gamma_p(0):=1$ and for $x\neq0$
\begin{align}
\Gamma_p(x):=\lim_{x_n\rightarrow x}\Gamma_p(x_n),\notag
\end{align}
where $(x_n)$ is a sequence of positive integers $p$-adically approaching to $x.$
For further details, see \cite{kob}. We now recall an important product formula for $p$-adic gamma functions from \cite[eq. (2.8)]{mccarthy-pacific}.
If $m\in\mathbb{Z}^+,$ 
$p\nmid m$  and $x=\frac{r}{p-1}$ with $0\leq r\leq p-1,$ then
\begin{align}\label{prod-1}
\prod_{h=0}^{m-1}\Gamma_p\left(\frac{x+h}{m}\right)=\omega(m^{(1-x)(1-p)})~\Gamma_p(x)\prod_{h=1}^{m-1}\Gamma_p\left(\frac{h}{m}\right).
\end{align}
From \cite{mccarthy-pacific} we also state that if $t\in\mathbb{Z}^{+}$ and $p\nmid t,$ then for $0\leq j\leq p-2$ we have
\begin{align}\label{prod-2}
\omega(t^{-tj})\Gamma_p\left(\left\langle\frac{-tj}{p-1}\right\rangle\right)\prod_{h=1}^{t-1}\Gamma_p\left(\frac{h}{t}\right)
=\prod_{h=1}^{t}\Gamma_p\left(\left\langle\frac{h}{t}-\frac{j}{p-1}\right\rangle\right).
\end{align}

\noindent Moreover, from \cite[Eq. (2.9)]{mccarthy-pacific} we note that 
\begin{align}\label{prod}
\Gamma_p(x)\Gamma_p(1-x)=(-1)^{a_0(x)},
\end{align}
where $a_0(x)\in\{1,2,\ldots, p\}$ such that $a_0(x)\equiv x\pmod{p}.$
Let $\pi \in \mathbb{C}_p$ be the fixed root of the polynomial $x^{p-1} + p$, which satisfies the congruence condition
$\pi \equiv \zeta_p-1 \pmod{(\zeta_p-1)^2}$. Then the following result is known as the Gross-Koblitz formula relating Gauss sum and $p$-adic gamma function as given below.

\begin{theorem}\cite[Gross-Koblitz formula]{gross}\label{gross-koblitz} For $j\in \mathbb{Z}$,
\begin{align}
g(\overline{\omega}^j)=-\pi^{(p-1)\langle\frac{j}{p-1} \rangle}\Gamma_p\left(\left\langle \frac{j}{p-1} \right\rangle\right).\notag
\end{align}
\end{theorem}

\noindent We now state a lemma that is obtained by simply applying the Gross-Koblitz formula.
\begin{lemma}
Let $p$ be an odd prime and $1\leq j\leq p-2.$ Then
\begin{align}\label{lemma-1}
\Gamma_p\left(\left\langle1-\frac{j}{p-1}\right\rangle\right)\Gamma_p\left(\left\langle\frac{j}{p-1}\right\rangle\right)
=-\omega^j(-1).
\end{align}
\end{lemma}

%\section{Hypergeometric functions and elliptic curves}

\subsection{Hypergeometric functions}
In this section, we mainly discuss three propositions that will be used to establish connections of the number of $\F_p$-points of $X_{\lambda}$ and the functions ${_3G_3}(\lambda)_p$ and ${_9G_9}(\lambda)_p.$ 
For $\lambda\in\F_p$, define the following character sum:

$$
C_p(\lambda):=\sum\limits_{j=0}^{p-2}g(\phi\overline{\omega}^{2j})g(\omega^j)^3g(\phi\overline{\omega}^j)\overline{\omega}^j (\lambda).
$$
The next lemma provides relation between the function ${_3F_2(\lambda)_p}$ and the character sum $C_p(\lambda).$

\begin{lemma}\label{3F2-Lemma}
For an odd prime $p$ and $\lambda\in\F_p$ we have 

\begin{align}
{_3F_2(\lambda)_p}&=\frac{\phi(\lambda-1)}{p^3(p-1)}C_p\left(\frac{\lambda}{(1-\lambda)^2}\right)+\frac{\phi(-1)}{p}\delta(1+\lambda)\notag\\
&+\delta(1-\lambda)\Delta(p)\left[{\chi_4\choose\phi} {\phi\overline{\chi_4}\choose\overline{\chi_4}}+ {\phi\chi_4\choose\phi} {\overline{\chi_4}\choose\phi\overline{\chi_4}}\right].\notag
\end{align}
\end{lemma}

\begin{proof}
We can write the function $_3F_2(\lambda)_p$ by applying Theorem 4.28 of \cite{greene} as given below. 
\begin{align}\label{3F2-Lemma-eq-1}
{_3F_2(\lambda)_p}=S_p-\frac{\phi(\frac{\lambda-1}{\lambda})}{p^2}+\frac{\phi(-1)}{p}\delta(1+\lambda)+\delta(1-\lambda)\Delta(p)\left[{\chi_4\choose\phi} {\phi\overline{\chi_4}\choose\overline{\chi_4}}+ {\phi\chi_4\choose\phi} {\overline{\chi_4}\choose\phi\overline{\chi_4}}\right],
\end{align}
where $$S_p:=\phi(1-\lambda)\frac{p}{p-1}\sum\limits_{\chi}{\phi\chi^2\choose\chi}{\phi\chi\choose\chi}{\phi\chi\choose\chi}\chi\left(\frac{-\lambda}{(1-\lambda)^2}\right).$$ By using the definition of binomial and then applying \eqref{Gauss-Jacobi-relation} and \eqref{inverse} we obtain

$$
S_p=\frac{\phi(\lambda(\lambda-1))}{p^2}+\frac{\phi(\lambda-1)}{p^3(p-1)}\sum\limits_{\chi}g(\phi\chi^2)g(\overline{\chi})^3g(\phi\chi)\chi\left(\frac{\lambda}{(1-\lambda)^2}\right).
$$
Finally, replacing $\chi$ by $\overline{\omega}^j$ in the above summation and then substituting the expression for $S_p$ into \eqref{3F2-Lemma-eq-1} we derive the result. 
\end{proof}

\begin{proposition}\label{Trace-1}
For an odd prime $p$ and $\lambda\in\F_p^{\times}$ we have 
$$
A_p(\lambda)=-\frac{\phi(-1)}{p(p-1)}C_p\left(\frac{1-\lambda}{\lambda^2}\right).
$$
\end{proposition}

\begin{proof}
Suppose that $\beta=\frac{1-\lambda}{\lambda^2}.$ Then we have 

\begin{align}
C_p(\beta)&=\sum\limits_{j=0}^{p-2}g(\phi\overline{\omega}^{2j}) g(\omega^j)^3 g(\phi\overline{\omega}^j)\overline{\omega}^j(\beta)\notag\\
&=\sum\limits_{x,y,z,u,v\neq0}\phi(xv) \zeta_p^{x+y+z+u+v} \sum\limits_{j=0}^{p-2}\omega^j\left(\frac{yzu}{\beta x^2v}\right).\notag
\end{align}

\noindent Now, by the orthogonality relation given in \eqref{orthogonal-1}, we claim that the above sum is non zero only if $yzu=\beta x^2v.$ Therefore, using this fact we may have

\begin{align}
C_p(\beta)&=(p-1)\sum\limits_{x,z,u,v\neq0}\phi(xv) \zeta_p^{x+\frac{\beta x^2v}{zu}+z+u+v}.\notag
\end{align}

Taking the transformations $u\mapsto ux, ~ v\mapsto vx ~ \text{and}~ z\mapsto zx$ we write

\begin{align}
C_p(\beta)&=(p-1)\sum\limits_{z,u,v\neq0}\phi(v)\sum\limits_{x\neq0} \zeta_p^{x\cdot f(z,u,v)},\notag
\end{align}

where $f(z,u,v):=1+\frac{\beta v}{zu}+z+u+v.$
Further simplifying we obtain
\begin{align}
C_p(\beta)&=(p-1)\sum\limits_{\substack{z,u,v\neq0\\ f(z,u,v)\neq0}}\phi(v)\sum\limits_{x\neq0} \zeta_p^{x\cdot f(z,u,v)} + (p-1)\sum\limits_{\substack{z,u,v\neq0\\ f(z,u,v)=0}}\phi(v)\sum\limits_{x\neq0} \zeta_p^{x\cdot f(z,u,v)}\notag\\
&=(p-1)g(\varepsilon)\sum\limits_{\substack{z,u,v\neq0\\ f(z,u,v)\neq0}}\phi(v) +(p-1)^2\sum\limits_{\substack{z,u,v\neq0\\ f(z,u,v)=0}}\phi(v)\notag\\
&=-(p-1)\sum\limits_{\substack{z,u,v\neq0\\ f(z,u,v)\neq0}}\phi(v) +(p-1)^2\sum\limits_{\substack{z,u,v\neq0\\ f(z,u,v)=0}}\phi(v)\notag\\
&=-(p-1)\sum\limits_{\substack{z,u,v\neq0\\ f(z,u,v)\neq0}}\phi(v)-(p-1)\sum\limits_{\substack{z,u,v\neq0\\ f(z,u,v)=0}}\phi(v)
+(p-1)\sum\limits_{\substack{z,u,v\neq0\\ f(z,u,v)=0}}\phi(v)\notag\\
& +(p-1)^2\sum\limits_{\substack{z,u,v\neq0\\ f(z,u,v)=0}}\phi(v)\notag\\
&=-(p-1)\sum\limits_{z,u,v\neq0}\phi(v)+p(p-1)\sum\limits_{\substack{z,u,v\neq0\\ f(z,u,v)=0}}\phi(v)\notag\\
&=p(p-1)\sum\limits_{\substack{z,u,v\neq0\\ f(z,u,v)=0}}\phi(v)\notag\\
&=p(p-1)\phi(-1)\sum\limits_{\substack{z,u\neq0\\ uz\neq-\beta}}\phi\left(1+\frac{\beta}{uz}\right)\phi(1+u+z)+p(p-1)\sum\limits_{\substack{u,z\neq0\\1+z+u=0}} \sum\limits_{v\neq0}\phi(v)\notag\\
&=p(p-1)\phi(-1)\sum\limits_{\substack{z,u\neq0\\ uz\neq-\beta}}\phi\left(1+\frac{\beta}{uz}\right)\phi(1+u+z)\notag\\
&=p(p-1)\phi(-1)\sum\limits_{\substack{z,u\neq0\\ uz\neq-\beta}}\phi(uz)\phi(uz+\beta)\phi(1+u+z)\notag\\
&=p(p-1)\phi(-1)\sum\limits_{z,u}\phi(uz)\phi(uz+\beta)\phi(1+u+z)\notag\\&=-p(p-1)\phi(-1)A_p(\lambda).\notag
\end{align}

\end{proof}

Now, we state three elementary lemmas that are going to use in the proofs of the next propositions.
\begin{lemma}\label{exponent-4}
For a prime $p\geq5$ and $0\leq j\leq p-2$ we have

$$1+\left\lfloor\frac{-1}{6}+\frac{2j}{p-1}\right\rfloor=\left\lfloor\frac{11}{12}+\frac{j}{p-1}\right\rfloor + \left\lfloor\frac{5}{12}+\frac{j}{p-1}\right\rfloor.$$
\end{lemma}

\begin{proof}
Since we have $\left\lfloor\frac{12j}{p-1}\right\rfloor=k,$ where $0\leq k\leq11.$ Therefore, considering $\left\lfloor\frac{12j}{p-1}\right\rfloor=k,$ it is easy to verify the result. 
\end{proof}

\begin{lemma}\label{exponent-1}
For a prime $p\geq5$ and $0\leq j\leq p-2$ we have

$$
\left\lfloor\frac{1}{2}+\frac{3j}{p-1}\right\rfloor=\left\lfloor\frac{1}{6}+\frac{j}{p-1}\right\rfloor + \left\lfloor\frac{5}{6}+\frac{j}{p-1}\right\rfloor+\left\lfloor\frac{1}{2}+\frac{j}{p-1}\right\rfloor.$$
\end{lemma}

\begin{proof}
Since for $0\leq j\leq p-2,$ we have $0\leq\frac{3j}{p-1}<3.$ So,
by considering $\frac{3j}{p-1}$ in the sub intervals $[0,1/2),[1/2, 1),\ldots, [2.5,3)$ of $[0,3),$ 
it is easy to show the result.
\end{proof}

\begin{lemma}\label{exponent-2}
For a prime $p\geq5$ and $0\leq j\leq p-2$ we have

\begin{align}
\left\lfloor\frac{1}{2}+\frac{6j}{p-1}\right\rfloor&=\left\lfloor\frac{1}{12}+\frac{j}{p-1}\right\rfloor + \left\lfloor\frac{5}{12}+\frac{j}{p-1}\right\rfloor+\left\lfloor\frac{7}{12}+\frac{j}{p-1}\right\rfloor\notag\\
&+\left\lfloor\frac{11}{12}+\frac{j}{p-1}\right\rfloor+\left\lfloor\frac{1}{4}+\frac{j}{p-1}\right\rfloor+\left\lfloor\frac{3}{4}+\frac{j}{p-1}\right\rfloor.\notag
\end{align}
\end{lemma}

\begin{proof}
Proof follows easily by using similar arguments as in the proof of Lemma \ref{exponent-1}.
\end{proof}

\begin{lemma}\label{exponent-3}
For a prime $p\geq5$ and $1\leq j\leq p-2$ we have

\begin{align}
\left\lfloor\frac{-3j}{p-1}\right\rfloor&=-1+\left\lfloor\frac{1}{3}-\frac{j}{p-1}\right\rfloor + \left\lfloor\frac{2}{3}-\frac{j}{p-1}\right\rfloor.\notag
\end{align}
\end{lemma}
\begin{proof}
The proof follows similarly as in the proof of Lemma \ref{exponent-4}.
\end{proof}

\begin{proposition}\label{3G3}
Let $p$ be an odd prime such that $p\equiv1\pmod{3}$, $\lambda\in\F_p\setminus\{1\}$ and $\psi_6=\omega^{\frac{p-1}{6}}$ be a character of order 6. Then we have 
$$
\psi_6(-1)p^2(p-1)\cdot{_3G_3}(\lambda)_p=C_p\left(\frac{\lambda}{(1-\lambda)^2}\right).
$$
\end{proposition}

\begin{proof}

Using the definition of ${_3G_3}(\lambda)_p$ and the fact that $\omega^j(-1)=(-1)^j$ as $p$ is odd, we may have

\begin{align}
{_3G_3}(\lambda)_p&=\frac{\psi_3\left(\frac{\lambda}{(1-\lambda)^2}\right)}{\beta_p(1-p)}\sum\limits_{j=0}^{p-2}\overline{\omega}^{j}~ \left(\frac{4\lambda}{(1-\lambda)^2}\right)(-p)^{\alpha_j}
 \Gamma_p\left(\left\langle\frac{11}{12}+\frac{j}{p-1}\right\rangle\right)\notag\\
&\times \Gamma_p\left(\left\langle\frac{5}{12}+\frac{j}{p-1}\right\rangle\right)\Gamma_p\left(\left\langle\frac{1}{6}+\frac{j}{p-1}\right\rangle\right)
 \Gamma_p\left(\left\langle\frac{1}{3}-\frac{j}{p-1}\right\rangle\right)^3,
\end{align}
where $$\alpha_j=-\left\lfloor\frac{11}{12}+\frac{j}{p-1}\right\rfloor-\left\lfloor\frac{5}{12}+\frac{j}{p-1}\right\rfloor-\left\lfloor\frac{1}{6}+\frac{j}{p-1}\right\rfloor
-3\left\lfloor\frac{1}{3}-\frac{j}{p-1}\right\rfloor$$ and $$\beta_p=\Gamma_p(1/3)^3\Gamma_p(11/12)\Gamma_p(5/12)\Gamma_p(1/6).$$

Now, applying the multiplication formula of $p$-adic gamma function given in \eqref{prod-1} for $m=2$ and $x=\left\langle\frac{-1}{6}+\frac{2j}{p-1}\right\rangle$ we have

\begin{align}\label{final-eq-1}
{_3G_3}(\lambda)_p&=\frac{\overline{\psi}_6(2)\Gamma_p(\frac{1}{2})\psi_3\left(\frac{\lambda}{(1-\lambda)^2}\right)}{(1-p)\beta_p}\sum\limits_{j=0}^{p-2}\overline{\omega}^{j}~ \left(\frac{\lambda}{(1-\lambda)^2}\right)(-p)^{\alpha_j}
 \Gamma_p\left(\left\langle\frac{-1}{6}+\frac{2j}{p-1}\right\rangle\right)\\
 &\times \Gamma_p\left(\left\langle\frac{1}{6}+\frac{j}{p-1}\right\rangle\right)
 \Gamma_p\left(\left\langle\frac{1}{3}-\frac{j}{p-1}\right\rangle\right)^3.\notag
\end{align}

Moreover, by using Lemma \ref{exponent-4} in the expression of $\alpha_j,$ we have 
\begin{align}
\alpha_j&=-1-\left\lfloor\frac{-1}{6}+\frac{2j}{p-1}\right\rfloor-\left\lfloor\frac{1}{6}+\frac{j}{p-1}\right\rfloor-3\left\lfloor\frac{1}{3}-\frac{j}{p-1}\right\rfloor\notag\\
&=-2+\left\langle\frac{-1}{6}+\frac{2j}{p-1}\right\rangle+3\left\langle\frac{1}{3}-\frac{j}{p-1}\right\rangle
+\left\langle\frac{1}{6}+\frac{j}{p-1}\right\rangle.\notag
\end{align}

 Now, if we use Gross-Koblitz formula (Theorem \ref{gross-koblitz}) on the right side of \eqref{final-eq-1}, then we have
 
 \begin{align}
{_3G_3}(\lambda)_p&=\frac{\overline{\psi}_6(2)\Gamma_p(\frac{1}{2})\psi_3\left(\frac{\lambda}{(1-\lambda)^2}\right)}{p^2(p-1)\beta_p}
 \sum\limits_{j=0}^{p-2}
g(\phi\overline{\omega}^{2j-\frac{2(p-1)}{3}})g(\omega^{j-\frac{p-1}{3}})^3g(\phi\overline{\omega}^{j-\frac{p-1}{3}})\overline{\omega}^{j}~ \left(\frac{\lambda}{(1-\lambda)^2}\right).\notag
 \end{align}
 Now, transforming $j\mapsto j+\frac{p-1}{3}$ we obtain

\begin{align}\label{eqn-10may-1}
{_3G_3}(\lambda)_p
&=\frac{\overline{\psi}_6(2)\Gamma_p(\frac{1}{2})}{p^2(p-1)\beta_p}\cdot C_p\left(\frac{\lambda}{(1-\lambda)^2}\right).
\end{align}
Now, we aim to show that $\frac{\overline{\psi}_6(2)\Gamma_p(\frac{1}{2})}{\beta_p}=\psi_6(-1).$ If we apply \eqref{prod-1} for $x=5/6$ and $m=2,$ then we have 
\begin{align}\label{eqn-10may}
\frac{\overline{\psi}_6(2)\Gamma_p(\frac{1}{2})}{\beta_p}=\frac{1}{\Gamma_p(1/3)^3\Gamma_p(1/6)\Gamma_p(5/6)}.
\end{align} 

\noindent Now, applying \eqref{prod} for $x=1/6$ we have 
$$\Gamma_p\left(\frac{1}{6}\right)\Gamma_p\left(\frac{5}{6}\right)=-\psi_6(-1).$$ 

\noindent Moreover, by using Theorem \ref{gross-koblitz} and \eqref{inverse} we have 
$$\Gamma_p\left(\frac{1}{3}\right)^3=\frac{g(\psi_3) g(\psi_3)}{g(\psi_3^2)}=J(\psi_3,\psi_3)=-1.$$ Substituting the above two values in \eqref{eqn-10may} we conclude that  $\frac{\overline{\psi}_6(2)\Gamma_p(\frac{1}{2})}{\beta_p}=\psi_6(-1).$ Finally, substituting this value in \eqref{eqn-10may-1} we complete the proof.

\end{proof}

\begin{proposition}\label{9G9}
Let $p$ be an odd prime such that $p\equiv2\pmod{3}$ and $\lambda\in\F_p\setminus\{1\}.$ Then we have 
$$
p(p-1)\phi(-1)\cdot{_9G_9}(\lambda)_p=C_p\left(\frac{\lambda}{(1-\lambda)^2}\right).
$$
\end{proposition}

\begin{proof}
For $\lambda\neq1,$ recall that
$$
C_p\left(\frac{\lambda}{(1-\lambda)^2}\right)=\sum\limits_{j=0}^{p-2} g(\phi\overline{\omega}^{2j})g(\omega^{j})^3g(\phi\overline{\omega}^{j})~\overline{\omega}^{j} \left(\frac{\lambda}{(1-\lambda)^2}\right).
$$

Since $\gcd(3,p-1)=1,$ so taking the transformation $j\rightarrow3j,$ we have
\begin{align}
C_p\left(\frac{\lambda}{(1-\lambda)^2}\right)&=\sum\limits_{j=0}^{p-2} g(\phi\overline{\omega}^{6j})g(\omega^{3j})^3g(\phi\overline{\omega}^{3j})~\overline{\omega}^{3j} \left(\frac{\lambda}{(1-\lambda)^2}\right).\notag
\end{align}
Now, applying Theorem \ref{gross-koblitz} on the right side of the above expression we obtain that

\begin{align}\label{eqn-30}
C_p\left(\frac{\lambda}{(1-\lambda)^2}\right)&=-\sum\limits_{j=0}^{p-2}\pi^{(p-1)e_j} \overline{\omega}^{j} \left(\frac{\lambda^3}{(1-\lambda)^6}\right)\Gamma_p\left(\left\langle\frac{1}{2}+\frac{6j}{p-1}\right\rangle\right) 
\Gamma_p\left(\left\langle\frac{-3j}{p-1}\right\rangle\right)^3\notag\\
&\times\Gamma_p\left(\left\langle\frac{1}{2}+\frac{3j}{p-1}\right\rangle\right),
\end{align}
where $e_j=\left\langle\frac{1}{2}+\frac{6j}{p-1}\right\rangle+3\left\langle\frac{-3j}{p-1}\right\rangle+\left\langle\frac{1}{2}+\frac{3j}{p-1}\right\rangle=1-\left\lfloor\frac{1}{2}+\frac{6j}{p-1}\right\rfloor-3\left\lfloor\frac{-3j}{p-1}\right\rfloor
-\left\lfloor\frac{1}{2}+\frac{3j}{p-1}\right\rfloor.$
Then by using Lemma \ref{exponent-1}, Lemma \ref{exponent-2} and Lemma \ref{exponent-3}, we may write

\begin{align}\label{eqn-26}
e_j&=1-\left\lfloor\frac{1}{12}+\frac{j}{p-1}\right\rfloor - \left\lfloor\frac{5}{12}+\frac{j}{p-1}\right\rfloor-\left\lfloor\frac{7}{12}+\frac{j}{p-1}\right\rfloor-\left\lfloor\frac{11}{12}+\frac{j}{p-1}\right\rfloor-\left\lfloor\frac{1}{4}+\frac{j}{p-1}\right\rfloor\notag\\
&-\left\lfloor\frac{3}{4}+\frac{j}{p-1}\right\rfloor-\left\lfloor\frac{1}{6}+\frac{j}{p-1}\right\rfloor-\left\lfloor\frac{5}{6}+\frac{j}{p-1}\right\rfloor-\left\lfloor\frac{1}{2}+\frac{j}{p-1}\right\rfloor-3\left\lfloor\frac{1}{3}-\frac{j}{p-1}\right\rfloor\notag\\
&-3\left\lfloor\frac{2}{3}-\frac{j}{p-1}\right\rfloor-3\left\lfloor\frac{-j}{p-1}\right\rfloor.
\end{align}

Now, applying the product formula given in \eqref{prod-1} we have 

\begin{align}\label{eqn-27}
\Gamma_p\left(\left\langle\frac{1}{2}+\frac{6j}{p-1}\right\rangle\right) &=\phi(6)\overline{\omega}^{6j}(6)\times \frac{\Gamma_p\left(\left\langle\frac{1}{12}+\frac{j}{p-1}\right\rangle\right) \Gamma_p\left(\left\langle\frac{1}{4}+\frac{j}{p-1}\right\rangle\right) \Gamma_p\left(\left\langle\frac{5}{12}+\frac{j}{p-1}\right\rangle\right) }{\Gamma_p(\frac{1}{6}) \Gamma_p(\frac{1}{3})\Gamma_p(\frac{1}{2})\Gamma_p(\frac{2}{3})\Gamma_p(\frac{5}{6})}\vspace{5mm}\notag\\
&\times \Gamma_p\left(\left\langle\frac{7}{12}+\frac{j}{p-1}\right\rangle\right) \Gamma_p\left(\left\langle\frac{3}{4}+\frac{j}{p-1}\right\rangle\right) \Gamma_p\left(\left\langle\frac{11}{12}+\frac{j}{p-1}\right\rangle\right),
\end{align}
and 

\begin{align}\label{eqn-28}
\Gamma_p\left(\left\langle\frac{1}{2}+\frac{3j}{p-1}\right\rangle\right) &=\phi(3)\overline{\omega}^{3j}(3)\times \frac{\Gamma_p\left(\left\langle\frac{1}{6}+\frac{j}{p-1}\right\rangle\right) \Gamma_p\left(\left\langle\frac{1}{2}+\frac{j}{p-1}\right\rangle\right)
\Gamma_p\left(\left\langle\frac{5}{6}+\frac{j}{p-1}\right\rangle\right)}{\Gamma_p(\frac{1}{3})\Gamma_p(\frac{2}{3})}.
\end{align}

Also, applying the product formula given in \eqref{prod-2} we have 

\begin{align}\label{eqn-29}
\Gamma_p\left(\left\langle\frac{-3j}{p-1}\right\rangle\right)=\omega^{3j}(3)\times \frac{\Gamma_p\left(\left\langle\frac{1}{3}-\frac{j}{p-1}\right\rangle\right) \Gamma_p\left(\left\langle\frac{2}{3}-\frac{j}{p-1}\right\rangle\right)
\Gamma_p\left(\left\langle1-\frac{j}{p-1}\right\rangle\right)}{\Gamma_p(\frac{1}{3})\Gamma_p(\frac{2}{3})}.
\end{align}

Substituting \eqref{eqn-26}, \eqref{eqn-27}, \eqref{eqn-28}, \eqref{eqn-29} into \eqref{eqn-30} we obtain that
\begin{align}
C_p\left(\frac{\lambda}{(1-\lambda)^2}\right)=-p(p-1)\phi(2)\cdot \kappa_p\cdot {_9G_9}(\lambda)_p,
\end{align}
where $\kappa_p=\Gamma_p(\frac{1}{12})\Gamma_p(\frac{11}{12})\Gamma_p(\frac{5}{12})\Gamma_p(\frac{7}{12})\Gamma_p(\frac{1}{4})\Gamma_p(\frac{3}{4}).$ Finally, using \eqref{prod-1} and \eqref{prod} we deduce that
$$\kappa_p=(-1)^{a_0(1/12)+a_0(5/12)+a_0(1/4)}=(-1)^{a_0(3/4)}=-\phi(-2).$$ 
Thus, we have

$$
C_p\left(\frac{\lambda}{(1-\lambda)^2}\right)=p(p-1)\phi(-1)\cdot {_9G_9}(\lambda)_p.
$$
This completes the proof.

\end{proof}

%%%%%%%%%%%%%%%%%%%
%%%%%%%%%%%%%%%%%

\section{Proof of theorems}
\begin{proof}[Proof of Theorem \ref{Frobenius-Moments}] First we prove even moments. If $m$ is even and $\lambda\neq0,$ then Proposition \ref{Trace-1} gives
$$
A_p(\lambda)^m=\frac{1}{p^m(p-1)^m}C_p\left(\frac{1-\lambda}{\lambda^2}\right)^m.
$$
Then replacing $\lambda$ by $1-\lambda$ in Lemma \ref{3F2-Lemma} and using the resultant expression for $\lambda\neq0,2$ we deduce that
\begin{align}\label{eqn-1}
A_p(\lambda)^m=p^{2m}{_3F_2}(1-\lambda)_p^m,
\end{align}
and $$A_p(2)^m=p^m(p\phi(-1)\cdot {_3F_2}(-1)_p-1)^m.$$

\noindent Using Theorem 6 (iii) of \cite{ono} we have that 
\begin{align}\label{eqn-2}
|p\cdot {_3F_2}(-1)_p|\leq3.
\end{align} Using this bound, for $p\rightarrow\infty,$ we have 
\begin{align}\label{eqn-3}
A_p(2)^m=o_m(p^{m+1}).
\end{align}
\noindent  Moreover, applying Theorem 4 of \cite{ono}, we obtain that as $p\rightarrow\infty$
\begin{align}\label{eqn-5}
p^{2m}\cdot {_3F_2}(1)_p^{m}=o_m(p^{m+1}).
\end{align}
Using \eqref{eqn-1} we have 

\begin{align}\label{eqn-4}
\sum\limits_{\lambda\in\F_p^{\times}}A_p(\lambda)^{m}&=\sum\limits_{\lambda\in\F_p^{\times}\backslash\{2\}}A_p(\lambda)^{m}+A_p(2)^{m}\notag\\
&=p^{2m}\sum\limits_{\lambda\in\F_p^{\times}\backslash\{2\}}{_3F_2}(1-\lambda)_p^{m}+A_p(2)^{m}\notag\\
&=p^{2m}\sum\limits_{\lambda\in\F_p}{_3F_2}(1-\lambda)_p^{m}-p^{2m}\cdot{_3F_2(1)_p}^{m}-p^{2m}\cdot{_3F_2(-1)_p}^{m}+A_p(2)^{m}\notag\\
&=p^{2m}\sum\limits_{\lambda\in\F_p}{_3F_2}(\lambda)_p^{m}-p^{2m}\cdot{_3F_2(1)_p}^{m}-p^{2m}\cdot{_3F_2(-1)_p}^{m}+A_p(2)^{m}.
\end{align}
The last equality is obtained by taking the transformation $\lambda\mapsto1-\lambda.$ Now, by using Theorem \ref{3F2-moments}, \eqref{eqn-2}, \eqref{eqn-3} and \eqref{eqn-5} we conclude the even moments. To obtain the odd moments, let $m$ be an odd positive integer. Then using similar arguments as in the even moments we have 

\begin{align}\label{eqn-15}
\sum\limits_{\lambda\in\F^{\times}_p}A_p(\lambda)^{m}&=\sum\limits_{\lambda\in\F_p^{\times}\backslash\{2\}}A_p(\lambda)^{m}+A_p(2)^{m}\notag\\
&=-p^{2m}\sum\limits_{\lambda\in\F_p\backslash\{2\}}\phi(\lambda){_3F_2}(1-\lambda)^{m}_p+A_p(2)^{m}\notag\\
&=-p^{2m}\sum\limits_{\lambda\in\F_p}\phi(\lambda){_3F_2}(1-\lambda)_p^{m}+p^{2m}\cdot\phi(2)\cdot{_3F_2(-1)}_p^{m}+A_p(2)^{m}\notag\\
&=-M_m+p^{2m}\cdot\phi(2)\cdot{_3F_2(-1)}_p^{m}+A_p(2)^{m},
\end{align} 
where
$$
M_m:=p^{2m}\sum\limits_{\lambda\in\F_p}\phi(\lambda){_3F_2}(1-\lambda)_p^{m}
$$
and 
\begin{align}\label{eqn-12may-1}
A_p(2)^m=-\phi(-1)p^m(p\cdot\phi(-2)\cdot{_3F_2(-1)_p}-1)^m.
\end{align}

Since
\begin{align}
\sum\limits_{\lambda\in\F_p}(1+\phi(\lambda)){_3F_2}(1-\lambda)_p^{m}&=\sum\limits_{\lambda\in\F_p}{_3F_2}(1-\lambda^2)_p^{m},\notag
\end{align}

so we have 

\begin{align}\label{eqn-14}
M_m=p^{2m}\sum\limits_{\lambda\in\F_p}{_3F_2}(1-\lambda^2)_p^{m}-p^{2m}\sum\limits_{\lambda\in\F_p}{_3F_2}(1-\lambda)_p^{m}.
\end{align}

Taking the transformation $\lambda\mapsto \lambda/2$ and using Theorem 4.2 (ii) of \cite{greene} we have 

\begin{align}\label{eqn-6}
\sum\limits_{\lambda\in\F_p}{_3F_2}(1-\lambda^2)_p^{m}=\sum\limits_{\lambda\in\F_p}{_3F_2}\left(\frac{4-\lambda^2}{4}\right)_p^{m}=\sum\limits_{\lambda\in\F_p\setminus\{\pm2\}}\phi(\lambda^2-4){_3F_2}\left(\frac{4}{4-\lambda^2}\right)_p^{m}.
\end{align}

For $\lambda^2\neq0,-1,4,$ if we apply Theorem 5 of \cite{ono}, then we have 

\begin{align}\label{eqn-7}
{_3F_2}\left(\frac{4}{4-\lambda^2}\right)_p=\phi(\lambda^4-4\lambda^2)\cdot\frac{(a_p^{\Cl}(\lambda^2))^2-p}{p^2}.
\end{align}
Now, using \eqref{eqn-7} in \eqref{eqn-6} we have
\begin{align}\label{eqn-8}
\sum\limits_{\lambda\in\F_p}{_3F_2}(1-\lambda^2)_p^{m}=\frac{1}{p^{2m}}\sum\limits_{\substack{\lambda\in\F_p\setminus\{0,\pm2\}\\ \lambda^2\neq-1}}(a_p^{\Cl}(\lambda^2)^2-p)^m+\phi(-5){_3F_2\left(\frac{4}{5}\right)_p^m}.
\end{align}
Now, by using Proposition \ref{moment-proposition}, \eqref{Hasse-bound} and \eqref{3F2bound} in \eqref{eqn-8} we deduce that 

\begin{align}\label{eqn-13}
p^{m-1}\sum\limits_{\lambda\in\F_p}{_3F_2}(1-\lambda^2)^m=\sum\limits_{i=0}^{m}(-1)^{i}{m\choose i}\frac{(2m-2i)!}{(m-i)!(m-i+1)!}+o_m(1).
\end{align}

If we use \eqref{eqn-13} on the first summation present on right side of \eqref{eqn-14} and Theorem \ref{3F2-moments} on the second summation of the right side of \eqref{eqn-14} for $m$ odd, then we obtain 

\begin{align}\label{eqn-16}
\sum\limits_{\lambda\in\F^{\times}_p}A_p(\lambda)^{m}=&-\sum\limits_{i=0}^{m}(-1)^{i}{m\choose i}\frac{(2m-2i)!p^{m+1}}{(m-i)!(m-i+1)!}+p^{2m}\cdot\phi(2)\cdot{_3F_2(-1)}_p^{m}+A_p(2)^{m}\notag\\
&+o_m(p^{m+1}).
\end{align}
Using \eqref{eqn-2} in \eqref{eqn-12may-1} we have 

$$
p^{2m}\cdot\phi(2)\cdot{_3F_2(-1)}_p^{m}+A_p(2)^{m}=o_m(p^{m+1}).
$$
Substituting this in \eqref{eqn-16} and then replacing $i$ by $m-i$ we deduce the result.

\end{proof}

To prove the distributions, we first fix some notation. Let $\mathbb{P}$ denote the set of primes. For each prime $p\in\mathbb{P},$ we define a function 
$$
f_p:\F_p\rightarrow[-3,3]
$$
and consider the probability space $(\Omega_p,\mathcal{F}_p,\mu_p),$ where $\Omega_p=\F_p,\mathcal{F}_p=\mathcal{P}(\F_p)$ and $\mu_p=\frac{|A|}{p}$ for all $A\in\mathcal{F}_p.$ 

\begin{proof}[Proof of Corollary \ref{Frobenius-Distribution}]
For the random variable $X_p=f_p,$ we have 
$$
\lim\limits_{p\rightarrow\infty}E(X_p^m)=\sum\limits_{i=0}^{m}(-1)^{i}{m\choose i}\frac{(2i)!}{(i)!(i+1)!}.
$$
Moreover, consider the probability space $(\Omega, \mathcal{F},\mu),$ where $\Omega=[-3,3],$ $\mathcal{F}$ is the collection of all Lebesgue measurable subsets of $\Omega$ and $\mu$ is the measure 
$\mu([a,b]):=\frac{1}{2\pi}\int\limits_{a}^bf(t) dt.$
Now, \begin{align}\label{eqn-13May-1}
E(X^m)=\frac{1}{2\pi}\int\limits_{-3}^3f(t)t^mdt=\frac{1}{2\pi}\int\limits_{-1}^3\sqrt{\frac{3-t}{1+t}}t^m dt.
\end{align} 
Now, following the arguments given in \cite[pg. 21 and pg. 22]{KHN} we conclude that 

$$
E(X^m)=\sum\limits_{i=0}^{m}(-1)^{i}{m\choose i}\frac{(2i)!}{(i)!(i+1)!}.
$$
Since the moment generating function has a positive radius of convergence, so the distribution of X is determined by its moments. Therefore,  $X_p$ converges in distribution to $X.$ This completes the proof.
\end{proof}

\begin{proof}[Proof of Theorem \ref{3G3-Moments}]
From Proposition \ref{3G3} we have for $\lambda\neq1$
$$
\psi_6(-1)p^2(p-1)\cdot {_3G_3}(\lambda)=C_p\left(\frac{\lambda}{(1-\lambda^2)}\right).
$$

Now, replacing $\lambda$ by $1-\lambda$ and applying Proposition \ref{Trace-1} we 
$$
p\cdot{_3G_3}(1-\lambda)=-A_p(\lambda).
$$
Finally, applying Theorem \ref{Frobenius-Moments} on the right side of the above we derive the result. This completes the proof.

\end{proof}

\begin{proof}[Proof of Corollary \ref{3G3-Distribution}]
The proof follows along similar lines as in the proof of Corollary \ref{Frobenius-Distribution}. The only changes are that we consider the set $\mathbb{P}$ as the collection of primes $p$ such that $p\equiv1\pmod{3}$ and replace $t$ by $-t$ in 
\eqref{eqn-13May-1}.
\end{proof}

\begin{proof}[Proof of Theorem \ref{9G9-Moments}] 
Replacing $\lambda$ by $1-\lambda$ in Proposition \ref{9G9} we have 

$$
p(p-1)\phi(-1)\cdot{_9G_9}(1-\lambda)_p=C_p\left(\frac{1-\lambda}{\lambda^2}\right).
$$

Now, by Proposition \ref{Trace-1} we may have 
$$
{_9G_9}(1-\lambda)_p=-A_p(\lambda).
$$
Hence, using Theorem \ref{Frobenius-Moments} on the right side of the above we derive the result. This completes the proof.

\end{proof}

\begin{proof}[Proof of Corollary \ref{9G9-Distribution}]
The proof follows along similar lines as in the proof of Corollary \ref{Frobenius-Distribution}. The only changes are that we consider the set $\mathbb{P}$ as the collection of primes $p$ such that $p\equiv2\pmod{3}$ and replace $t$ by $-t$ in 
\eqref{eqn-13May-1}.
\end{proof}

\section{Acknowledgements}

The first author is supported by Science and Engineering Research Board [SRG/2023/000930] and the second author is supported by Science and Engineering Research Board [CRG/2023/003037].

%%%%%%%%%%%%%%%%%%%%%%
%%%%%%%%%%%%%%%%%%%%%%
%%%%%%%%%%%%%%%%%%%%%%

%%%%%%%%%%%%%%%%%%
%%%%%%%%%%%%%%%%%%

\end{document}